\documentclass[12pt,reqno]{amsart}
\usepackage{amssymb,amsmath,mathrsfs,xcolor}
\usepackage[left=2.9cm,right=2.9cm,top=2.8cm,bottom=2.8cm]{geometry}

\newcommand{\N}{\mathbb{N}}
\newcommand{\R}{\mathbb{R}}
\newcommand{\dx}{\, {\rm d} x}
\newcommand{\ds}{\, {\rm d} \sigma}
\newcommand{\cone}{{\rm int}(C_+)}
\renewcommand{\epsilon}{\varepsilon}
\renewcommand{\phi}{\varphi}
\DeclareMathOperator{\esssup}{\rm ess\,sup}

\newtheorem{defi}{Definition}[section]
\newtheorem{prop}[defi]{Proposition}
\newtheorem{lemma}[defi]{Lemma}
\newtheorem{thm}[defi]{Theorem}

\newtheorem{rmk}[defi]{Remark}

\numberwithin{equation}{section}

\begin{document}
\title[Multiple solutions to quasi-linear elliptic Robin systems]{Multiple solutions to\\ quasi-linear elliptic Robin systems}
\author[U. Guarnotta]{U. Guarnotta}
\address[U. Guarnotta]{Dipartimento di Matematica e Informatica, Universit\`a degli Studi di Palermo, Via Archirafi 34, 90123 Palermo, Italy}
\email{umberto.guarnotta@unipa.it}
\author[S.A. Marano]{S.A. Marano}
\address[S.A. Marano]{Dipartimento di Matematica e Informatica, Universit\`a degli Studi di Catania, Viale A. Doria 6, 95125 Catania, Italy}
\email{marano@dmi.unict.it}
\author[A. Moussaoui]{A. Moussaoui}
\address[A. Moussaoui]{LMA, Faculty of Exact Sciences, A. Mira Bejaia University, 06000 Bejaia, Algeria}
\email{abdelkrim.moussaoui@univ-bejaia.dz}
\maketitle

\begin{abstract}
Two opposite constant-sign solutions to a non-variational $p$-Laplacian system with Robin boundary conditions are obtained via sub-super-solution techniques. A third nontrivial one comes out by means of topological degree arguments. 
\end{abstract}

{
\let\thefootnote\relax
\footnote{Corresponding author: S.A. Marano.}
}
\setcounter{footnote}{0}

\section{Introduction}
Multiplicity of solutions to boundary-value problems for elliptic systems looks an interesting topic, both from a theoretical point of view and as regards applications. In fact, even without convective terms, the problem doesn't usually have a variational structure (which forces very special reactions) and so abstract results different from mountain-pass like theorems are needed. The most common attempt is using fixed point results or some kind of topological degree. Recently, this was accomplished in the paper \cite{MMPe}, that investigates the Dirichlet problem
\begin{equation}\label{prob0}
\left\{
\begin{alignedat}{2}
-\Delta_{p_1} u_1 &= f_1(x,u_1,u_2)\quad &&\mbox{in}\;\;\Omega, \\
-\Delta_{p_2} u_2 &= f_2(x,u_1,u_2)\quad &&\mbox{in} \;\;\Omega, \\
u_1 = u_2&=0\quad &&\mbox{on} \;\;\partial \Omega,
\end{alignedat}
\right.
\end{equation}
where $p_i\in(1,+\infty)$, the symbol $\Delta_{p_i}$ denotes the $p_i$-Laplacian, i.e.,
\begin{equation*}
\Delta_{p_i}u:={\rm div}(|\nabla u|^{p_i-1}\nabla u)\quad\forall\, u\in W^{1,p_i}(\Omega),
\end{equation*}
$\Omega$ is a bounded domain in $\R^N$, $N\geq 2$, with smooth boundary $\partial\Omega$, and $f_i:\Omega\times\R^2\to\R$, $i=1,2$, satisfy Carath\'eodory's conditions. Two opposite constant-sign solutions were obtained via the sub-super-solution method for systems; see, e.g., \cite[Section 5.5]{CLM}. Then appropriately using the Leray-Schauder degree gave a third nontrivial (precisely, nodal) solution. Chapter 6 of the monograph \cite{Mo} represents an excellent introduction to these topics, while \cite{DM} provides a singular version of \cite{MMPe}. Recall that \eqref{prob0} is called singular when for instance $f_i:\Omega\times(\R^+)^2\to\R^+$ and $\lim_{s_i\to 0^+}f_i(x,s_1,s_2)=+\infty$; vide the very recent survey \cite{GLM}.

A quick search in the Mathematical Reviews shows that, until today, there are about forty papers on Dirichlet problem, sixteen dealing with the Neumann case, and only one, namely \cite{Am}, devoted to Robin boundary conditions. From a technical point of view, most of them employ variational methods, whereas a few are based on topological degree arguments. This work investigates the non-variational Robin system
\begin{equation}\label{prob}
\tag{P}
\left\{
\begin{alignedat}{2}
-\Delta_{p_1} u_1 &= f_1(x,u_1,u_2) \quad &&\mbox{in} \;\; \Omega, \\
-\Delta_{p_2} u_2 &= f_2(x,u_1,u_2) \quad &&\mbox{in} \;\; \Omega, \\
|\nabla u_1|^{p_1-2} \nabla u_1 \cdot \nabla \nu &= -\beta_1 |u_1|^{p_1-2}u_1  \quad &&\mbox{on} \;\; \partial \Omega, \\
|\nabla u_1|^{p_2-2} \nabla u_2 \cdot \nabla \nu &= -\beta_2 |u_2|^{p_2-2}u_2  \quad &&\mbox{on} \;\; \partial \Omega,
\end{alignedat}
\right.
\end{equation}
where $\nu(x)$ indicates the outward unit normal vector to $\partial\Omega$ at its point $x$ while $\beta_i>0$. We adapt the approach of \cite{MMPe}, which basically requires four tools:
\begin{itemize}
\item some natural properties (cf. Proposition \ref{properties}) of the involved differential operator;
\item a non-existence result (Proposition \ref{nonex}) for a family of auxiliary Robin problems;
\item the anti-maximum principle (Proposition \ref{antimax}) with Robin boundary conditions;
\item a sub-super-solution result (Theorem \ref{subsuperthm}).
\end{itemize}
Under suitable assumptions (see Section \ref{results}), two opposite constant-sign smooth solutions to \eqref{prob} are obtained in Theorem \ref{constantsign}. A third nontrivial solution then comes out through topological degree arguments; cf. Theorem \ref{thirdsol}. Unfortunately, we were not able to show that it is nodal.

Multiplicity results for elliptic Robin \textit{equations} with complete sign information on the solutions have been established by several authors. Usually, critical point theory, truncation-perturbation techniques, and Morse identity produce at least three non-zero solutions; two of them have opposite constant sign and the third one is nodal. A recent reference work on this topic is \cite{PW}.
\section{Preliminaries}
Let $(U,\|\cdot\|)$ be a real Banach space and let $U^*$ be its topological dual, with duality brackets $\langle\cdot, \cdot\rangle$. An operator $A:U\to U^*$ is said to be:
\begin{itemize}
\item \emph{bounded} if it maps bounded sets into bounded sets.
\item \emph{coercive} provided $\displaystyle{\lim_{\Vert u\Vert\to+\infty}} \frac{\langle A(u),u\rangle}{\Vert u\Vert}=+\infty$.
\item \emph{pseudo-monotone} if $u_n\rightharpoonup u$ in $U$ and $\displaystyle{\limsup_{n\to+\infty}}\langle A(u_n),u_n-u\rangle\leq 0$ force
$$\displaystyle{\liminf_{n\to+\infty}}\langle A(u_n),u_n-v\rangle\geq\langle A(u),u-v\rangle\quad\forall\, v\in U.$$
\item \emph{of type $(\mathrm{S})_+$} provided
\begin{equation*}
u_n\rightharpoonup u\;\;\mbox{in $U$,}\;\;\limsup_{n\to+\infty}\langle A(u_n),u_n-u\rangle\leq 0\implies u_n\to u\;\;\mbox{in $U$.}   
\end{equation*}
\end{itemize}
Recall that 
\begin{equation*}
\mbox{$A$ continuous and of type $(\mathrm{S})_+$}\implies
\mbox{$A$ pseudo-monotone.}    
\end{equation*}
Actually, demi-continuity suffices; vide \cite[Lemma 6.7]{F}. Moreover, one has (see, e.g., \cite[Theorem 2.99]{CLM})
\begin{thm}\label{thmbrezis}
\label{brezis} Let $U$ be a reflexive Banach space. If $A:U\to U^*$ is bounded, continuous, coercive, and pseudo-monotone then $A(U)=U^*$.
\end{thm}
Henceforth, $\Omega$ will denote a bounded domain of the real Euclidean $N$-space $(\R^N, |\cdot|)$, $N\geq 2$, with smooth boundary $\partial\Omega$, on which we will employ the $(N-1)$-dimensional Hausdorff measure $\sigma$.

If $X(\Omega)$ is a measurable real-valued function space on $\Omega$ and $u,v\in X(\Omega)$ then, by definition, $u<v$ means $u(x)<v(x)$ almost everywhere in $\Omega$. Analogously for $u\leq v$, etc. Put 
\begin{equation*}
X(\Omega)_+:=\left\{u\in X(\Omega): u\geq 0\right\}   
\end{equation*}
as well as
\begin{equation*}
[u,v]:=\{ w\in X(\Omega): u\leq w\leq v\}
\end{equation*}
once $u\leq v$. To shorten notation, set $C_+:=C^1(\overline{\Omega})_+$. An easy computation yields
\begin{equation*}
\cone:=\{ u\in C^1(\overline{\Omega}):u(x)>0\;\;\forall\, x\in \overline{\Omega}\}.
\end{equation*}
Let $p\in [1,+\infty]$. We denote by $p^{\prime}$ the conjugate exponent of $p$, while
\begin{equation*}
\Vert u\Vert _{p}:=\left\{ 
\begin{array}{ll}
\left(\int_{\Omega}|u|^{p}\dx\right)^{1/p} &\text{if }p<+\infty, \\ 
\phantom{} & \\ 
\underset{\Omega}{\esssup}\,|u| &\text{otherwise.}
\end{array}
\right.
\end{equation*}
The symbol $|E|$ stands for the $N$-dimensional Lebesgue measure of the set $E\subseteq\R^{N}$ and
\begin{equation*}
a\vee b:=\max\{a,b\},\;\; a^\pm:=(\pm a)\vee 0,\;\; a\wedge b:=\min\{a,b\}\;\;
\quad\forall\, a,b\in\R.
\end{equation*}
Finally, $g:\Omega\times\R^2\to\R$ is called a Carath\'{e}odory function when:
\begin{itemize}
\item $x\mapsto g(x,s_1,s_2)$ turns out measurable for every $(s_1,s_2)\in\R^2$;
\item $(s_1,s_2)\mapsto g(x,s_1,s_2)$ is continuous for almost all $x\in\Omega$.
\end{itemize}
\section{Auxiliary results}
Let $p\in (1,+\infty)$ and let $\beta>0$. Consider the operator $A_p:W^{1,p}(\Omega)\to W^{1,p}(\Omega)^*$ given by
\begin{equation*}
\langle A_p(u),\psi\rangle:=\int_\Omega |\nabla u|^{p-2}\nabla u \nabla\psi \dx +\beta \int_{\partial\Omega} |u|^{p-2}u\psi\ds\quad\forall\, u,\psi\in W^{1,p}(\Omega).
\end{equation*}
If we define
\begin{equation}\label{equivnorm}
\mathscr{E}_p(u):=\langle A_p(u),u \rangle = \int_\Omega |\nabla u|^p \dx + \beta \int_{\partial \Omega} |u|^p \ds\, , \quad u \in W^{1,p}(\Omega),
\end{equation}
then $\mathscr{E}_p(\cdot)^{1/p}$ is a norm on $W^{1,p}(\Omega)$  equivalent to the usual one $\Vert\cdot\Vert_{1,p}$; cf. \cite[Remark 1.54]{MMP}. Some basic facts concerning $A_p$ are collected in the next result.
\begin{prop}\label{properties}
$({\rm a}_1)$ $A_p$ is bounded, continuous, strictly monotone, and of type ${\rm (S)_+}$.\\
$({\rm a}_2)$ $A_p$ is bijective and  $A_p^{-1}:W^{1,p}(\Omega)^*\to W^{1,p}(\Omega)$ is continuous.
\end{prop}
\begin{proof} Boundedness is trivial, while continuity comes from the same property of the trace operator, for which we refer to \cite[Theorem 1.39]{MMP}. Since strict monotonicity is a straightforward consequence of the vector inequalities in \cite[Chapter 12]{Lind}, it remains to verify that $A_p$ fulfills condition ${\rm (S)_+}$. We evidently have $A_p=B_p+C_p$, where
\begin{equation}
\label{splitting}
\begin{split}
\langle B_p(u),\psi \rangle := \int_\Omega |\nabla u|^{p-2} \nabla u \nabla \psi \dx, \quad \langle C_p(u),\psi \rangle := \int_{\partial \Omega} |u|^{p-2} u \psi \ds.
\end{split}
\end{equation}
By \cite[Proposition 12.14]{MMP} the operator $B_p$ is of type ${\rm (S)_+}$, while $C_p$ turns out monotone. Thus, the assertion follows from \cite[Proposition 2.138(c)]{GP}. This shows $({\rm a}_1)$. \\ 
After noting that
\begin{equation*}
\langle A_p(u),u\rangle=\mathcal{E}_p(u)\geq c\Vert u\Vert_{1,p}^p\quad\forall\, u\in W^{1,p}(\Omega),
\end{equation*}
Theorem \ref{thmbrezis} and $({\rm a}_1)$ directly yield the bijectivity of $A_p$; see also \cite[Theorem 26.A]{Zeidler}. Now, pick a sequence $\{f_n\}\subseteq W^{1,p}(\Omega)^*$ such that $f_n\to f$ in $W^{1,p}(\Omega)^*$ and write  $u_n:=A_p^{-1}(f_n)$, $u:=A_p^{-1}(f)$. One evidently has:
\begin{equation}\label{weakformulations}
\langle A_p(u_n),\psi\rangle=\langle f_n,\psi\rangle\quad\forall\, n\in\N,\; \psi\in W^{1,p}(\Omega);
\end{equation}
\begin{equation}\label{weakformulation}
\langle A_p(u),\psi\rangle=\langle f,\psi\rangle\quad \forall\,\psi\in W^{1,p}(\Omega).
\end{equation}
Through \eqref{weakformulations}, besides \eqref{equivnorm}, we obtain
\begin{equation*}
\mathscr{E}_p(u_n)\leq\|f_n\|_{W^{1,p}(\Omega)^*}\|u_n\|_{1,p}
\leq c\,\mathscr{E}_p(u_n)^{1/p},\;\;n\in\N,
\end{equation*}
because $\{f_n\}$ turns out bounded. Therefore, $\{u_n\}\subseteq W^{1,p}(\Omega)$ enjoys the same property and, taking a sub-sequence if necessary, $u_n\rightharpoonup\hat u$ in $W^{1,p}(\Omega)$. Accordingly, by \eqref{weakformulations} again,
\begin{equation*}
\lim_{n\to\infty}\langle A_p(u_n),u_n-\hat u\rangle=\lim_{n\to\infty}\langle f_n,u_n-\hat u\rangle = 0.
\end{equation*}
Since $A_p$ is of type ${\rm (S)_+}$, this entails $u_n\to \hat u$ in $W^{1,p}(\Omega)$ and, a fortiori,
\begin{equation*}
\lim_{n\to+\infty}\Vert\nabla u_n-\nabla\hat u\Vert_p=0,\;\;
\lim_{n\to+\infty}\Vert u_n-\hat u\Vert_{L^p(\partial\Omega)}=0   
\end{equation*}
(recall that the trace operator is continuous).
%We get
%\begin{equation}
%\label{boundedness}
%\begin{alignedat}{2}
%&\{|\nabla u_n|^{p-2}\nabla u_n\} \quad &&\mbox{is bounded in} \;\; L^{p'}(\Omega), \\
%&\{|u_n|^{p-2}u_n\} \quad &&\mbox{is bounded in} \;\; L^{p'}(\partial\Omega).
%\end{alignedat}
%\end{equation}
%Moreover, by \cite[Theorem 4.9]{B}, we have $\nabla u_n\to \nabla v$ a.e. in $\Omega$ and $u_n \to v$ a.e. on $\partial\Omega$, so
%\begin{equation}
%\label{pointwiseconv}
%\begin{alignedat}{2}
%|\nabla u_n|^{p-2}\nabla u_n &\to |\nabla w|^{p-2}\nabla w \quad &&\mbox{a.e. in} \;\; \Omega, \\
%|u_n|^{p-2}u_n &\to |w|^{p-2}w \quad &&\mbox{a.e. on} \;\; \partial\Omega.
%\end{alignedat}
%\end{equation}
%Applying \cite[Exercise 4.16]{B} together with \eqref{boundedness}--\eqref{pointwiseconv} yields
%\begin{equation}
%\label{weakconv}
%\begin{alignedat}{2}
%|\nabla u_n|^{p-2}\nabla u_n &\rightharpoonup |\nabla w|^{p-2}\nabla w \quad &&\mbox{in} \;\; L^{p'}(\Omega), \\
%|u_n|^{p-2}u_n &\rightharpoonup |w|^{p-2}w \quad &&\mbox{on} \;\; L^{p'}(\partial\Omega).
%\end{alignedat}
%\end{equation}
Thanks to \cite[Theorem 2.76]{MMP} we thus arrive at
\begin{equation}\label{nemytskii}
\begin{alignedat}{2}
|\nabla u_n|^{p-2}\nabla u_n &\to |\nabla\hat u|^{p-2}\nabla\hat u\quad &&\mbox{in $L^{p'}(\Omega)$,}\\
|u_n|^{p-2}u_n &\to |\hat u|^{p-2}\hat u\quad &&\mbox{in $L^{p'}(\partial\Omega)$,}
\end{alignedat}
\end{equation}
whence
\begin{equation*}
\langle B_p(u_n),\psi\rangle\to\langle B_p(\hat u),\psi\rangle\;\;\mbox{and}\;\; \langle C_p(u_n),\psi\rangle\to\langle C_p(\hat u),\psi\rangle\;\;\forall\,\psi\in W^{1,p}(\Omega).
\end{equation*}
On the other hand, from $f_n\to f$ in $W^{1,p}(\Omega)^*$ it follows $\langle f_n,\psi\rangle\to\langle f,\psi\rangle$ whatever $\psi\in W^{1,p}(\Omega)$. Thus, \eqref{weakformulations}--\eqref{weakformulation} easily lead to
\begin{equation*}
\langle A_p(\hat u),\psi\rangle=\langle f,\psi\rangle=\langle A_p(u),\psi\rangle,\;\; \psi \in W^{1,p}(\Omega).
\end{equation*}
By the strict monotonicity of $A_p$, we finally have $\hat u=u$. Summing up, $u_n\to u$, namely $A_p^{-1}(f_n)\to A_p^{-1}(f)$, in $W^{1,p}(\Omega)$, as desired.
\end{proof}
We indicate with $(\lambda_p,\phi_p)$, omitting the dependence on $\beta$ when no confusion can arise, the pair of first eigenvalue-eigenfunction of the problem
\begin{equation}\label{eigen}
\left\{
\begin{alignedat}{2}
-\Delta_p\phi &=\lambda |\phi|^{p-2}\phi\quad &&\mbox{in}\;\;\Omega,\\
|\nabla\phi|^{p-2}\nabla\phi\cdot\nabla\nu &=-\beta|\phi|^{p-2}\phi\quad &&\mbox{on}\;\; \partial\Omega.
\end{alignedat}
\right.
\end{equation}
Recall that $\phi_p\in\cone$, without loss of generality one can assume $\|\phi_p\|_p=1$, and 
\begin{equation}\label{varchar}
\lambda_p=\inf\left\{ \mathscr{E}_p(u):\, \|u\|_p = 1 \right\},
\end{equation}
the infimum being attained at $\phi_p$. Denote by $\hat{\lambda}_p$ the second eigenvalue of \eqref{eigen}.

The next two results are chiefly patterned after Proposition 9.64 and Theorem 9.67 in \cite{MMP}, respectively.
\begin{prop}\label{nonex}
Let $h\in L^\infty(\Omega)_+\setminus\{0\}$. Then the problem
\begin{equation}\label{nonexprob}
\tag{${\rm P}_{\lambda_p}$}
\left\{
\begin{alignedat}{2}
-\Delta_p u &= \lambda_p |u|^{p-2}u + h(x) \quad &&\mbox{in} \;\; \Omega, \\
|\nabla u|^{p-2} \nabla u\cdot\nabla\nu &= -\beta |u|^{p-2}u  \quad &&\mbox{on} \;\; \partial\Omega
\end{alignedat}
\right.
\end{equation}
has no weak solution.
\end{prop}
\begin{proof}
Suppose there exists $u\in W^{1,p}(\Omega)$ solving \eqref{nonexprob}. We first show that $u\in\cone$. Testing with $-u^-$ entails
\begin{equation}\label{negenergy}
\mathscr{E}_p(u^-) = \lambda_p \|u^-\|_p^p - \int_\Omega hu^- \dx \leq \lambda_p \|u^-\|_p^p.
\end{equation}
Via \eqref{varchar} one has $u^-=k\phi_p$ for some $k\geq 0$. If $k>0$ then $u^-\in\cone$. Therefore, $\int_\Omega hu^- \dx>0$ and \eqref{negenergy} becomes
\begin{equation*}
\mathscr{E}_p(u^-) < \lambda_p \|u^-\|_p^p,
\end{equation*}
which contradicts \eqref{varchar}. So $k=0$, namely $u \geq 0$ in $\Omega$. The conditions on $h$ force $u \not\equiv 0$. Thus, standard results from nonlinear regularity theory \cite{L} and the strong maximum principle \cite{PS} yield $u \in \cone$.

Applying Picone's identity \cite[Theorem 1.1]{AH} to $\phi_p$ and $u+\epsilon$, with $\epsilon>0$, we get
\begin{equation*}
\begin{split}
0 &\leq \|\nabla \phi_p\|_p^p - \int_\Omega |\nabla u|^{p-2}\nabla u \nabla \left( \frac{\phi_p^p}{(u+\epsilon)^{p-1}} \right) \dx \\
&= \|\nabla \phi_p\|_p^p + \beta \int_{\partial\Omega} \left(\frac{u}{u+\epsilon} \right)^{p-1} \phi_p^p \ds -\lambda_p \int_\Omega \left(\frac{u}{u+\epsilon} \right)^{p-1} \phi_p^p \dx \\
&\quad - \int_\Omega \frac{h\phi_p^p}{(u+\epsilon)^{p-1}} \dx.
\end{split}
\end{equation*}
As $\epsilon\searrow 0$, Beppo Levi's theorem and the properties of $\phi_p$ lead to
\begin{equation*}
0\leq\mathscr{E}_p(\phi_p)-\lambda_p-\int_\Omega\frac{h\phi_p^p}{u^{p-1}} \dx =-\int_\Omega \frac{h\phi_p^p}{u^{p-1}} \dx<0.
\end{equation*}
This contradiction completes the proof.
\end{proof}

\begin{prop}\label{antimax}
Assume $h\in L^\infty(\Omega)_+\setminus\{0\}$. Then there exists $\delta>0$ such that, for every $\mu\in (\lambda_p,\lambda_p+ \delta)$, all solutions of the problem
\begin{equation}\label{antimaxprob}
\tag{${\rm P}_\mu$}
\left\{
\begin{alignedat}{2}
-\Delta_p u &= \mu |u|^{p-2}u + h(x) \quad &&\mbox{in} \;\; \Omega, \\
|\nabla u|^{p-2}\nabla u\cdot\nabla\nu &= -\beta |u|^{p-2}u \quad &&\mbox{on}\;\; \partial \Omega
\end{alignedat}
\right.
\end{equation}
belong to $-\cone$.
\end{prop}

\begin{proof}
If the conclusion were false we might construct two sequences $\{\mu_n\}\subseteq\R^+$ and $\{u_n\} \subseteq W^{1,p}(\Omega)$ fulfilling $\mu_n\searrow\lambda_p$, each $u_n$ solves $({\rm P}_{\mu_n})$, as well as
\begin{equation}\label{abs1}
u_n\not\in-\cone\quad\forall\, n\in\N.    
\end{equation}
Let us first verify that $\{u_n\}$ is unbounded in $L^\infty(\Omega)$. Supposing
$\sup_{n\in\N}\Vert u_n\Vert_\infty<+\infty$, standard results from regularity theory \cite{L} yield
$\sup_{n\in\N}\Vert u_n\Vert_{C^{1,\alpha}(\overline{\Omega})}<+\infty$.
Hence, thanks to Ascoli-Arzelà's theorem, $u_n\to u$ in $C^1(\overline{\Omega})$, where a sub-sequence is considered when necessary. As $n\to+\infty$, this entails that $u$ solves \eqref{nonexprob}, against Proposition \ref{nonex}. 

Define $M_n:=\|u_n\|_\infty$ and observe that, up to sub-sequences, $M_n\to +\infty$. Dividing both sides of $({\rm P}_{\mu_n})$ by $M_n^{p-1}$ we then obtain
\begin{equation}\label{homantimaxprob}
\left\{
\begin{alignedat}{2}
-\Delta_p w_n &= \mu_n |w_n|^{p-2}w_n + \frac{h(x)}{M_n^{p-1}} \quad &&\mbox{in} \;\; \Omega, \\
|\nabla w_n|^{p-2} \nabla w_n \cdot \nabla \nu &= -\beta |w_n|^{p-2}w_n  \quad &&\mbox{on} \;\; \partial \Omega,
\end{alignedat}
\right.
\end{equation}
where $w_n:=\frac{u_n}{M_n}$. Since $\Vert w_n\Vert_\infty=1$ for all $n\in\N$, like before, nonlinear regularity and Ascoli-Arzelà's theorem produce $w_n \to w$ in $C^1(\overline{\Omega})$. From \eqref{homantimaxprob} it follows that $w$ solves \eqref{eigen}, because $\mu_n\to\lambda_p$ while $M_n\to+\infty$.  Moreover, $\|w\|_\infty = 1$ implies $w\not\equiv 0$, whence either $w\in \cone$ or $w\in -\cone$. The condition $w\in \cone$ forces $w_n \in \cone$ for any $n$ big enough. Due to \eqref{homantimaxprob} again one evidently has
\begin{equation*}
\left\{
\begin{alignedat}{2}
&-\Delta_p w_n = \lambda_p |w_n|^{p-2}w_n + (\mu_n-\lambda_p) w_n^{p-1} + \frac{h(x)}{M_n^{p-1}} \quad &&\mbox{in} \;\; \Omega, \\
&|\nabla w_n|^{p-2} \nabla w_n \cdot \nabla \nu = -\beta |w_n|^{p-2}w_n  \quad &&\mbox{on} \;\; \partial \Omega,
\end{alignedat}
\right.
\end{equation*}
but this contradicts Proposition \ref{nonex}. So, assume $w\in-\cone$. Then $w_n \in-\cone$ and, a fortiori, $u_n\in-\cone$ for all sufficiently large $n$, which, bearing in mind \eqref{abs1}, is absurd.
\end{proof}
The section ends by proving a sub-super-solution theorem. To shorten notation, write
\begin{equation*}
X^{p_1,p_2}(\Omega):=W^{1,p_1}(\Omega)\times W^{1,p_2}(\Omega).
\end{equation*}
This space is equipped with the norm
\begin{equation*}
(u_1,u_2)\in X^{p_1,p_2}(\Omega)\mapsto \mathscr{E}_{p_1}(u_1)^{1/{p_1}} +\mathscr{E}_{p_2}(u_2)^{1/{p_2}},    
\end{equation*}
equivalent to the usual one $\Vert u_1\Vert_{1,p_1}+\Vert u_2\Vert_{1,p_2}$.

We say that $(\underline{u}_1,\underline{u}_2),(\overline{u}_1,\overline{u}_2)\in 
X^{p_1,p_2}(\Omega)$ form a sub-super-solution pair for \eqref{prob} provided 
$(\underline{u}_1,\underline{u}_2)\leq(\overline{u}_1,\overline{u}_2)$ and
\begin{equation}\label{subsuper}
\begin{split}
\langle A_{p_1} \underline{u}_1,\psi_1 \rangle \leq \int_\Omega f_1(x,\underline{u}_1,v_2)\psi_1 \dx,\;\;
\langle A_{p_2} \underline{u}_2,\psi_2 \rangle \leq \int_\Omega f_2(x,v_1,\underline{u}_2)\psi_2 \dx, \\
\langle A_{p_1} \overline{u}_1,\psi_1 \rangle \geq \int_\Omega f_1(x,\overline{u}_1,v_2)\psi_1 \dx,\;\;
\langle A_{p_2} \overline{u}_2,\psi_2 \rangle \geq \int_\Omega f_2(x,v_1,\overline{u}_2)\psi_2 \dx \\
\end{split}
\end{equation}
for all $(\psi_1,\psi_2)\in X^{p_1,p_2}(\Omega)_+$, $(v_1,v_2)\in [\underline{u}_1,\overline{u}_1]\times [\underline{u}_2,\overline{u}_2]$. The set
\begin{equation*}
\mathcal{C}:=[\underline{u}_1,\overline{u}_1]\times [\underline{u}_2,\overline{u}_2]
\end{equation*}
is often called \textit{trapping region}.
\begin{thm}
%[Sub-super-solution theorem]
\label{subsuperthm}
Let $(\underline{u}_1,\underline{u}_2)$, $(\overline{u}_1,\overline{u}_2)$ be a sub-super-solution pair. Suppose there exists $\mu_i>0$ such that
\begin{equation}\label{boundfi}
|f_i(x,s_1,s_2)|\leq\mu_i\;\;\forall\, x\in\Omega,\; (s_1,s_2)\in  [\underline{u}_1(x),\overline{u}_1(x)]\times [\underline{u}_2(x),\overline{u}_2(x)], \end{equation}
$i=1,2$. Then problem \eqref{prob} admits a solution $(u_1,u_2)\in [\underline{u}_1, \overline{u}_1]\times [\underline{u}_2,\overline{u}_2]$.
\end{thm}
\begin{proof}
The truncation operator $T_i:W^{1,p_i}(\Omega) \to W^{1,p_i}(\Omega)$ given by
\begin{equation*}
T_i(u)(x) = \left\{
\begin{alignedat}{2}
&\underline{u}_i(x) \quad &&\mbox{if} \;\; u(x) \leq \underline{u}_i(x), \\
&u_i(x) \quad &&\mbox{if} \;\; \underline{u}_i(x) \leq u(x) \leq \overline{u}_i(x), \\
&\overline{u}_i(x) \quad &&\mbox{if} \;\; u(x) \geq \overline{u}_i(x). \\
\end{alignedat}
\right.
\end{equation*}
is bounded continuous \cite[Lemma 2.89]{CLM} and monotone, as a simple computation shows. Hence, setting 
\begin{equation*}
\langle F_i(u_1,u_2),\psi \rangle = \int_\Omega f_i(x,T_1(u_1),T_2(u_2))\psi \dx\;\;\forall\,(u_1,u_2)\in X^{p_1,p_2}(\Omega),\;\psi\in W^{1,p_i}(\Omega)
\end{equation*}
yields a continuous map $F_i:X^{p_1,p_2}(\Omega)\to W^{1,p_i}(\Omega)^*$ such that
\begin{equation}\label{boundFi}
\sup\left\{\Vert F_i(u_1,u_2)\Vert_{W^{1,p_i}(\Omega)^*}: (u_1,u_2)\in X^{p_1,p_2}(\Omega)\right\}<+\infty   
\end{equation}
because of \eqref{boundfi}.
%We also consider the operators $\hat{A}_{p_i},\hat{B}_{p_i},\hat{C}_{p_i}:W^{1,p_i}(\Omega) \to (W^{1,p_i}(\Omega))^*$, $i=1,2$, defined as
%%
%\begin{equation*}
%\hat{B}_{p_i} = B_{p_i} \circ T_i, \quad \hat{C}_{p_i} = C_{p_i} \circ T_i, \quad \hat{A}_{p_i}:=\hat{B}_{p_i}+\hat{C}_{p_i},
%\end{equation*}
%%
%being $B_p,C_p$ defined in \eqref{splitting}. Taking into account Proposition \ref{properties}, it turns out that $\hat{A}_{p_i},\hat{B}_{p_i},\hat{C}_{p_i}$ are bounded and continuous. \\
Now, consider the operator $\Phi:X^{p_1,p_2}(\Omega)\to X^{p_1,p_2}(\Omega)^*$ defined by
\begin{equation*}
\langle \Phi(u_1,u_2),(\psi_1,\psi_2) \rangle = \langle A_{p_1}(u_1) - F_1(u_1,u_2),\psi_1 \rangle + \langle A_{p_2}(u_2) - F_2(u_1,u_2),\psi_2 \rangle.
\end{equation*}
An elementary argument ensures that $(u_1,u_2)\in\Phi^{-1}(0)\cap\mathcal{C}$ iff $(u_1,u_2)\in\mathcal{C}$ solves \eqref{prob}. Accordingly, the conclusion rewrites as $0\in\Phi(\mathcal{C})$.
%Indeed, if $(u_1,u_2)$ solves \eqref{prob}, then $\langle A_{p_1}(u_1) - F_1(u_1,u_2),\psi_1 \rangle = \langle A_{p_2}(u_2) - F_2(u_1,u_2),\psi_2 \rangle = 0$ for all $(\psi_1,\psi_2) \in W^{1,p_1}(\Omega) \times W^{1,p_2}(\Omega)$, so $\Phi(u_1,u_2)=0$. Conversely, if $\Phi(u_1,u_2)=0$ then one can choose $(\psi_1,\psi_2)=(\psi,0)$, with any $\psi \in W^{1,p_1}(\Omega)$, to verify the first equation of \eqref{prob}; the same argument works also for the second equation.

Via \eqref{boundfi} and the embedding $W^{1,p_i}(\Omega)\hookrightarrow L^1(\Omega)$  we infer
\begin{equation*}
\begin{split}
\langle\Phi(u_1,u_2),(u_1,u_2)\rangle&\geq\mathscr{E}_{p_1}(u_1)+\mathscr{E}_{p_2}(u_2)
-\max\{\mu_1,\mu_2\}(\|u_1\|_1+\|u_2\|_1) \\
&\geq\mathscr{E}_{p_1}(u_1)+\mathscr{E}_{p_2}(u_2)-c\max\{\mu_1,\mu_2\}
(\mathscr{E}_{p_1}(u_1)^{1/{p_1}}+\mathscr{E}_{p_2}(u_2)^{1/{p_2}})
\end{split}
\end{equation*}
for some $c:=c(p_1,p_2,\beta_1,\beta_2,\Omega)>0$. This entails
\begin{equation*}
\langle\Phi(u_1,u_2),(u_1,u_2)\rangle\to+\infty\quad\mbox{as}\quad
\mathscr{E}_{p_1}(u_1)^{1/{p_1}}+\mathscr{E}_{p_2}(u_2)^{1/{p_2}}\to +\infty,   
\end{equation*}
namely $\Phi$ is coercive. Thanks to Proposition \ref{properties} and \cite[Theorem 2.109]{CLM}, each operator $A_{p_i}-F_i$ fulfills condition ${\rm (S)_+}$. Thus, gathering \eqref{boundFi} and Proposition \ref{properties} together, we see that $\Phi$ turns out bounded, continuous, coercive, and of type ${\rm (S)_+}$, whence  pseudo-monotone. By Theorem \ref{thmbrezis}, there exists $(u_1,u_2)\in X^{p_1,p_2}(\Omega)$ satisfying $\Phi(u_1,u_2) = 0$. Let us finally verify that $(u_1,u_2)\in\mathcal{C}$. With this aim, inequalities \eqref{subsuper} written for $(\psi_1,\psi_2)=((u_1-\overline{u}_1)^+, (u_2-\overline{u}_2)^+)$ lead to
\begin{equation*}
\begin{split}
0 &= \langle A_{p_1}(u_1) - F_1(u_1,u_2),(u_1-\overline{u}_1)^+ \rangle + \langle A_{p_2}(u_2) - F_2(u_1,u_2),(u_2-\overline{u}_2)^+) \rangle \\
&\geq \langle A_{p_1}(u_1) - F_1(u_1,u_2),(u_1-\overline{u}_1)^+ \rangle + \langle A_{p_2}(u_2) - F_2(u_1,u_2),(u_2-\overline{u}_2)^+) \rangle \\
&\quad - \langle A_{p_1}(\overline{u}_1) - F_1(\overline{u}_1,u_2),(u_1-\overline{u}_1)^+ \rangle - \langle A_{p_2}(\overline{u}_2) - F_2(u_1,\overline{u}_2),(u_2-\overline{u}_2)^+) \rangle \\
&= \langle A_{p_1}(u_1)-A_{p_1}(\overline{u}_1),(u_1-\overline{u}_1)^+ \rangle + \langle A_{p_2}(u_2)-A_{p_2}(\overline{u}_2),(u_2-\overline{u}_2)^+ \rangle.
\end{split}
\end{equation*}
The strict monotonicity of $A_{p_i}$ (cf. Proposition \ref{properties}) forces $u_i\leq\overline{u}_i$. Since a similar reasoning produces $\underline{u}_i\leq u_i$, we actually have $(u_1,u_2)\in\mathcal{C}$.
\end{proof}

\begin{rmk}
Like in \cite{CM}, one can establish a more general sub-super-solution theorem for Robin problems, encompassing reactions with convection terms.
\end{rmk}
\section{Main results}\label{results}
Let $f_1,f_2:\Omega\times\R^2\to\R$ be two Carath\'{e}odory functions. The following hypotheses will be posited. 
\begin{itemize}
\item[${\rm (H_1)}$] There exist $k_{1,-},k_{2,-}<0<k_{1,+},k_{2,+}$ such that
\begin{equation*}
f_1(x,k_{1,+},s_2)\vee f_2(x,s_1,k_{2,+})\leq 0\quad\forall\, (x,s_1,s_2)\in\Omega \times [0,k_{1,+}]\times [0,k_{2,+}]
\end{equation*}
 and
\begin{equation*}
f_1(x,k_{1,-},s_2)\wedge f_2(x,s_1,k_{2,-})\geq 0\quad\forall\, (x,s_1,s_2)\in\Omega\times [k_{1,-},0] \times [k_{2,-},0].
\end{equation*}
\item[${\rm (H_2)}$] There are constants $\eta_i > \lambda_{p_i}$, $i=1,2$, fulfilling
\begin{equation*}
\liminf_{s_i\to 0^+}\frac{f_i(x,s_1,s_2)}{|s_i|^{p_i-2}s_i}\geq \eta_i\quad\mbox{uniformly in } (x,s_j)\in\Omega\times [0,k_{j,+}],\; j\neq i,
\end{equation*}
as well as
\begin{equation*}
\liminf_{s_i\to 0^-}\frac{f_i(x,s_1,s_2)}{|s_i|^{p_i-2}s_i}\geq\eta_i\quad\mbox{uniformly in } (x,s_j)\in\Omega\times [k_{j,-},0],\; j\neq i.
\end{equation*}
\item[${\rm (H_3)}$] To each $\rho>0$ there corresponds $\mu_i>0$ such that
\begin{equation*}
|f_i(x,s_1,s_2)|\leq\mu_i,\; i=1,2,\;\;\forall\, (x,s_i,s_j)\in\Omega\times [-\rho,\rho]\times\R,\; j \neq i.
\end{equation*}
\item[${\rm (H_4)}$] There exist constants $\theta_i >\lambda_{p_i}$, $i=1,2$, satisfying
\begin{equation*}
\lim_{s_i\to-\infty}\frac{f_i(x,s_1,s_2)}{|s_i|^{p_i-2}s_i} = 0\quad\mbox{and}\quad
\lim_{s_i\to+\infty}\frac{f_i(x,s_1,s_2)}{|s_i|^{p_i-2}s_i} =\theta_i
\end{equation*}
uniformly with respect to $(x,s_j)\in\Omega\times\R$, $j\neq i$.
\end{itemize}
Proofs below are chiefly patterned after the corresponding ones of \cite{MMPe}. The first result gives two constant-sign smooth solutions.

\begin{thm}\label{constantsign}
Under $({\rm H_1})$--$({\rm H_3})$, problem \eqref{prob} admits a positive solution $(u_{1,+},u_{2,+})$ and a negative solution $(u_{1,-},u_{2,-})$ in $C^{1,\alpha}(\overline{\Omega})^2$, $\alpha\in (0,1]$. Moreover, for $i=1,2$,
\begin{equation}
\label{trapped}
k_{i,-} \leq u_{i,-}<0<u_{i,+} \leq k_{i,+} \quad \mbox{in} \;\; \Omega,
\end{equation}
where $k_{i,\pm}$ stem from $({\rm H_1})$.
\end{thm}

\begin{proof}
We seek a positive solution, since the negative one can be obtained similarly. Set $\overline{u}_i = k_{i,+}$, $i=1,2$. Condition $({\rm H_1})$ yields
\begin{equation*}
\begin{split}
-\Delta_{p_1}\overline{u}_1\geq 0\geq f_1(x,\overline{u}_1,s_2)\;\;\forall\, (x,s_2)\in \Omega\times [0,\overline{u}_2],\\
-\Delta_{p_2}\overline{u}_2\geq 0\geq f_2(x,s_1,\overline{u}_2)\;\;\forall\, (x,s_1)\in \Omega\times [0,\overline{u}_1].
\end{split}
\end{equation*}
By $({\rm H_2})$ there exist $\delta>0$ and $\xi_i\in(\lambda_{p_i},\eta_i)$, $i=1,2$, such that
\begin{equation*}
\begin{split}
f_1(x,s_1,s_2)\geq\xi_1 s_1^{p_1-1}\;\;\mbox{in}\;\;\Omega\times (0,\delta]\times [0,k_{2,+}],\\
f_2(x,s_1,s_2)\geq\xi_2 s_2^{p_2-1}\;\;\mbox{in}\;\;\Omega\times [0,k_{1,+}]\times (0,\delta].
\end{split}
\end{equation*}
Hence, if $\underline{u}_i:=\epsilon\phi_{p_i}$ then
\begin{equation*}
\begin{split}
-\Delta_{p_1}\underline{u}_1\leq\lambda_{p_1}\underline{u}_1^{p_1-1}
<\xi_1\underline{u}_1^{p_1-1}\leq f_1(x,\underline{u}_1,s_2)\;\;\forall\, (x,s_2)\in\Omega \times [0,\overline{u}_2],\\
-\Delta_{p_2}\underline{u}_2\leq\lambda_{p_2}\underline{u}_2^{p_2-1}
<\xi_2\underline{u}_2^{p_2-1}\leq f_2(x,s_1,\underline{u}_2)\;\;\forall\, (x,s_1)\in\Omega \times [0,\overline{u}_1]
\end{split}
\end{equation*}
once $\epsilon<\delta(\max_{i=1,2}\|\phi_{p_i}\|_\infty)^{-1}$. Taking a smaller $\epsilon$ when necessary entails $\underline{u}_i\leq \overline{u}_i$, $i=1,2$. So, all the assumptions of Theorem \ref{subsuperthm} hold true, and we obtain a solution $(u_{1,+},u_{2,+})\in X^{p_1,p_2}(\Omega)$ to \eqref{prob} fulfilling \eqref{trapped}. Finally, standard results from nonlinear regularity theory ensure that $(u_{1,+},u_{2,+})\in C^{1,\alpha}(\overline{\Omega})^2$ for some $\alpha\in (0,1]$.
\end{proof}

A third solution comes out via the next lemmas, whose proofs are based on topological degree arguments.

Pick any $\delta>0$ and choose $\xi_i\in(\lambda_{p_i},\lambda_{p_i}+\delta)$, $i=1,2$. Given $t\in [0,1]$, consider the family of problems
\begin{equation}\label{topprob1}
\tag{${\rm \tilde{P}}_t$}
\left\{
\begin{alignedat}{2}
-\Delta_{p_1} u_1 &=\tilde{f}_{1,t}(x,u_1,u_2)\;\; &&\mbox{in}\;\;\Omega,\\
-\Delta_{p_2} u_2 &=\tilde{f}_{2,t}(x,u_1,u_2)\;\; &&\mbox{in} \;\;\Omega,\\
|\nabla u_1|^{p_1-2}\nabla u_1\cdot\nabla\nu &=-\beta_1 |u_1|^{p_1-2}u_1\;\; &&\mbox{on}\;\;\partial\Omega,\\
|\nabla u_1|^{p_2-2} \nabla u_2\cdot\nabla\nu &=-\beta_2 |u_2|^{p_2-2}u_2\;\; &&\mbox{on}\;\;\partial\Omega,
\end{alignedat}
\right.
\end{equation}
where
\begin{equation*}
\tilde{f}_{i,t}(x,s_1,s_2):=tf_i(x,s_1,s_2)+(1-t)(1+\xi_i (s_i^+)^{p_i-1}),\;\; (x,s_1,s_2)\in\Omega\times\R^2.
\end{equation*}
Let $\mathcal{H}:[0,1]\times \overline{B}_{\tilde{R}}\to L^{p_1}(\Omega)\times L^{p_2}(\Omega)$ be defined by setting
\begin{equation}\label{homo1}
\mathcal{H}(t,u_1,u_2):= (u_1-A_{p_1}^{-1}\tilde{f}_{1,t}(\cdot,u_1,u_2),u_2-A_{p_2}^{-1}\tilde{f}_{2,t}(\cdot,u_1,u_2))
\end{equation}
for all $(t,u_1,u_2)\in[0,1]\times\overline{B}_{\tilde{R}}$, with $\tilde{R}>0$ and
\begin{equation}\label{ball}
B_R:=\{(u_1,u_2)\in L^{p_1}(\Omega)\times L^{p_2}(\Omega):\|u_1\|_{p_1}+ \|u_2\|_{p_2} < R\},\;\; R>0.
\end{equation}
Thanks to Proposition \ref{properties} we see that $\mathcal{H}$ is a homotopy. In fact, if $e_i:W^{1,p_i}(\Omega)\to L^{p_i}(\Omega)$ denotes the usual embedding map then $e_i\circ A_{p_i}^{-1}:W^{1,p_i}(\Omega)^*\to L^{p_i}(\Omega)$ is compact continuous. Since $\tilde{f}_{i,t}(x,s_1,s_2)$ has a $(p_i-1)$-linear growth in $s_i$ uniformly with respect to $s_j$, $j\neq i$, while the adjoint $e_i^*$ of $e_i$ turns out continuous, the Nemytskii-like operator
\begin{equation*}
N_{\tilde{f}_{i,t}}:L^{p_1}(\Omega)\times L^{p_2}(\Omega)\to L^{p_i'}(\Omega) \stackrel{e^*_i}\to W^{1,p_i}(\Omega)^*     
\end{equation*}
turns out well defined, bounded, and continuous \cite[Theorem 2.75]{MMP}. Thus, the function from $L^{p_1}(\Omega)\times L^{p_2}(\Omega)$ into itself given by
\begin{equation*}
(u_1,u_2)\mapsto (A_{p_1}^{-1}\tilde{f}_{1,t}(\cdot,u_1,u_2),A_{p_2}^{-1}\tilde{f}_{2,t}(\cdot,u_1,u_2))
\end{equation*}
is compact continuous, as desired.
\begin{lemma}\label{degree1}
If $({\rm H_3})$--$({\rm H_4})$ are satisfied then:
\begin{enumerate}
\item For all $t\in [0,1]$, the Leray-Schauder degree $\deg(\mathcal{H}(t,\cdot,\cdot), B_{\tilde{R}},0)$ turns out well defined  provided $\tilde{R}>0$ is sufficiently large.
\item For every $\delta>0$ small enough one has
\begin{equation}\label{degest1}
\deg(\mathcal{H}(1,\cdot,\cdot),B_{\tilde{R}},0) = \deg(\mathcal{H}(0,\cdot,\cdot),B_{\tilde{R}},0) = 0.
\end{equation}
\end{enumerate}
\end{lemma}
\begin{proof}
Since $W^{1,p_i}(\Omega)\hookrightarrow L^{p_i}(\Omega)$, assertion (1) follows once we show that the solution set of \eqref{topprob1} is bounded uniformly with respect to $t\in [0,1]$. In fact, otherwise, there would exist two sequences $\{t_n\}\subseteq [0,1]$, $\{(u_{1,n},u_{2,n})\}\subseteq X^{p_1,p_2}(\Omega)$ fulfilling
\begin{equation}\label{proptn}
\mbox{$(u_{1,n},u_{2,n})$ solves $({\rm \tilde{P}}_{t_n})$ for all $n\in\N$,}\;\;t_n\to t,\;\;\|u_{1,n}\|_{p_1} + \|u_{2,n}\|_{p_2} \to +\infty.
\end{equation}
Up to sub-sequences, we may suppose $M_n:=\|u_{1,n}\|_{p_1}\to+\infty$. If $v_n:= \frac{u_{1,n}}{M_n}$ then
\begin{equation}\label{vn}
-\Delta_{p_1} v_n = \frac{t_n}{M_n^{p_1-1}}f_1(x,M_n v_n,u_{2,n}) + (1-t_n) (M_n^{1-p_1} + \xi_1 (v_n^+)^{p_1-1}).
\end{equation}
Denote by $w_n$ the right-hand side of \eqref{vn}. Recalling that $\|v_n\|_{p_1}=1$ and exploiting $({\rm H_3})$--$({\rm H_4})$ easily produce
\begin{equation*}
\int_\Omega |w_n|^{p_1'}\dx\leq c_1\int_\Omega (1+|v_n|^{p_1-1})^{p_1'}\dx
\leq c_1 2^{p_1'-1}(|\Omega|+1),\;\; n\in\N,
\end{equation*}
for appropriate $c_1>0$. Hence, $\{w_n\}$ turns out bounded in $L^{p_1'}(\Omega)$. Thanks to the compact embedding $L^{p_1'}(\Omega)\hookrightarrow W^{1,p_1}(\Omega)^*$, there exists $w\in W^{1,p_1}(\Omega)^*$ such that
\begin{equation}\label{dualconv}
w_n\to w\;\;\mbox{in}\;\; W^{1,p_1}(\Omega)^*.
\end{equation}
Now, test \eqref{vn} with $v_n$ and use \eqref{dualconv} to arrive at
\begin{equation}\label{energyest}
\mathscr{E}_{p_1}(v_n)=\int_\Omega w_n v_n\dx\leq \|w_n\|_{W^{1,p}(\Omega)^*} \mathscr{E}_{p_1}(v_n)^{1/{p_1}}\leq
c_2\mathscr{E}_{p_1}(v_n)^{1/{p_1}}\;\;\forall\, n\in\N,
\end{equation}
i.e., $\{v_n\}\subseteq W^{1,p_1}(\Omega)$ is bounded. Passing to a sub-sequence if necessary, we thus get $v_n\rightharpoonup v$ in $W^{1,p_1}(\Omega)$, whence 
\begin{equation*}
\lim_{n \to \infty} \langle A_{p_1}(v_n),v_n-v \rangle = \lim_{n \to \infty} \langle w_n,v_n-v \rangle = 0
\end{equation*}
because of \eqref{dualconv}. From Proposition \ref{properties} it finally follows $v_n\to v$ in $W^{1,p_1}(\Omega)$. This entails
\begin{equation*}
%\label{v}
\left\{
\begin{alignedat}{2}
-\Delta_{p_1} v &= (t\theta_1 + (1-t)\xi_1)(v^+)^{p_1-1} \quad &&\mbox{in} \;\; \Omega,\\
|\nabla v|^{p_1-2} \nabla v \cdot \nabla \nu &= -\beta_1 |v|^{p_1-2}v  \quad &&\mbox{on} \;\; \partial \Omega,
\end{alignedat}
\right.
\end{equation*}
where \eqref{vn}, \eqref{proptn}, and $({\rm H_4})$ have been used. Test the first equation with $-v^-$ to achieve $v\geq 0$. So, $v$ is an eigenfunction associated with the eigenvalue $t\theta_1 + (1-t)\xi_1$. The conditions $\theta_1,\xi_1> \lambda_{p_1}$ imply $t\theta_1 + (1-t)\xi_1>\lambda_{p_1}$. Accordingly, $v$ must be nodal \cite[Proposition 5.5]{Le}, which contradicts the information $v\geq 0$. Thus, $\deg(\mathcal{H}(t,\cdot,\cdot),B_{\tilde{R}}, 0)$ turns out well defined for all $t\in [0,1]$ provided $\tilde{R}>0$ is large enough.

Let us now come to assertion (2). The homotopy invariance property of the Leray-Schauder degree directly yields
\begin{equation*}
deg(\mathcal{H}(1,\cdot,\cdot),B_{\tilde{R}},0)= \deg(\mathcal{H}(0,\cdot,\cdot),B_{\tilde{R}},0).
\end{equation*}
If $(u_1,u_2)\in X^{p_1,p_2}(\Omega)$ solves problem $({\rm \tilde{P}}_0)$, i.e.,
\begin{equation}\label{tildeP0}
\tag{${\rm \tilde{P}}_0$}
\left\{
\begin{alignedat}{2}
-\Delta_{p_1} u_1 &= 1+\xi_1 (u_1^+)^{p_1-1} \quad &&\mbox{in} \;\; \Omega, \\
-\Delta_{p_2} u_2 &= 1+\xi_2 (u_2^+)^{p_2-1} \quad &&\mbox{in} \;\; \Omega, \\
|\nabla u_1|^{p_1-2} \nabla u_1 \cdot \nabla \nu &= -\beta_1 |u_1|^{p_1-2}u_1  \quad &&\mbox{on} \;\; \partial \Omega, \\
|\nabla u_1|^{p_2-2} \nabla u_2 \cdot \nabla \nu &= -\beta_2 |u_2|^{p_2-2}u_2  \quad &&\mbox{on} \;\; \partial \Omega,
\end{alignedat}
\right.
\end{equation}
then testing the first equation with $-u_1^-$ produces $u_1\geq 0$. On the other hand, through Proposition \ref{antimax} one has $u_1<0$ once $\delta$ is sufficiently small. Hence, no solution to \eqref{tildeP0} can exist, forcing $\deg(\mathcal{H} (0,\cdot,\cdot),B_{\tilde{R}},0) = 0$.
\end{proof}
Pick $\xi_i\in(\lambda_{p_i},\hat{\lambda}_{p_i})$, $i=1,2$. Given $t\in [0,1]$, consider the family of problems
\begin{equation}\label{topprob2}
\tag{${\rm \hat{P}}_t$}
\left\{
\begin{alignedat}{2}
-\Delta_{p_1} u_1 &= \hat{f}_{1,t}(x,u_1,u_2) \quad &&\mbox{in} \;\; \Omega, \\
-\Delta_{p_2} u_2 &= \hat{f}_{2,t}(x,u_1,u_2) \quad &&\mbox{in} \;\; \Omega, \\
|\nabla u_1|^{p_1-2} \nabla u_1 \cdot \nabla \nu &= -\beta_1 |u_1|^{p_1-2}u_1  \quad &&\mbox{on} \;\; \partial \Omega, \\
|\nabla u_1|^{p_2-2} \nabla u_2 \cdot \nabla \nu &= -\beta_2 |u_2|^{p_2-2}u_2  \quad &&\mbox{on} \;\; \partial \Omega,
\end{alignedat}
\right.
\end{equation}
where
\begin{equation*}
\hat{f}_{i,t}(x,s_1,s_2):=tf_i(x,s_1,s_2) + (1-t)\xi_i (s_i^+)^{p_i-1},\;\; (x,s_1,s_2)\in\Omega\times\R^2.
\end{equation*}
Let $\mathcal{K}:[0,1]\times\overline{B}_{\hat{R}}\to L^{p_1}(\Omega)\times L^{p_2}(\Omega)$ be defined by
\begin{equation}\label{homo2}
\mathcal{K}(t,u_1,u_2):= (u_1-A_{p_1}^{-1}\hat{f}_{1,t}(\cdot,u_1,u_2),u_2-A_{p_2}^{-1}\hat{f}_{2,t}(\cdot,u_1,u_2))
\end{equation}
for all $(t,u_1,u_2)\in [0,1]\times\overline{B}_{\hat{R}}$, with $\hat{R}>0$ and $B_{\hat{R}}$ as in \eqref{ball}. Like before, the function $\mathcal{K}$ is a homotopy.
\begin{lemma}\label{degree2}
If $({\rm H_3})$--$({\rm H_4})$ are satisfied then:
\begin{enumerate}
\item[(3)] For all $t\in [0,1]$, the Leray-Schauder degree $\deg(\mathcal{K}(t,\cdot,\cdot), B_{\hat{R}},0)$ turns out well defined  provided $\hat{R}>0$ is big enough.
\item[(4)] One has 
\begin{equation}\label{degest2}
\deg(\mathcal{K}(1,\cdot,\cdot),B_{\hat{R}},0) = \deg(\mathcal{K}(0,\cdot,\cdot),B_{\hat{R}},0) = 1.   
\end{equation}
\end{enumerate}
\end{lemma}
\begin{proof}
The proof goes on exactly as that of Lemma \ref{degree1}, except for the second equality in \eqref{degest2}. Thus, suppose $(u_1,u_2)\in X^{p_1,p_2}(\Omega)$ solves $({\rm \hat{P}_0})$, namely
\begin{equation}
\label{hatP0}
\tag{${\rm \hat{P}}_0$}
\left\{
\begin{alignedat}{2}
-\Delta_{p_1} u_1 &= \xi_1 (u_1^+)^{p_1-1} \quad &&\mbox{in} \;\; \Omega, \\
-\Delta_{p_2} u_2 &= \xi_2 (u_2^+)^{p_2-1} \quad &&\mbox{in} \;\; \Omega, \\
|\nabla u_1|^{p_1-2} \nabla u_1 \cdot \nabla \nu &= -\beta_1 |u_1|^{p_1-2}u_1  \quad &&\mbox{on} \;\; \partial \Omega, \\
|\nabla u_1|^{p_2-2} \nabla u_2 \cdot \nabla \nu &= -\beta_2 |u_2|^{p_2-2}u_2  \quad &&\mbox{on} \;\; \partial \Omega.
\end{alignedat}
\right.
\end{equation}
Testing the first equation with $-u_1^-$ reveals that $u_1\geq 0$ and $u_1$ is an eigenfunction related to the eigenvalue $\xi_1\in(\lambda_{p_i},\hat{\lambda}_{p_i})$. Hence, $u_1=0$. A similar argument furnishes $u_2=0$. So, \eqref{hatP0} admits only the trivial solution and, by definition, $\deg(\mathcal{K}(0,\cdot,\cdot),B_{\hat{R}},0) = 1$.
\end{proof}
We are in a position now to establish the existence of a third non-zero smooth solution.
\begin{thm}\label{thirdsol}
Under $({\rm H_1})$--$({\rm H_4})$, problem \eqref{prob} possesses a solution $(u_{1,0},u_{2,0})\in C^{1,\alpha}(\overline{\Omega})^2$ such that
\begin{equation}\label{nontrapped}
(u_{1,0},u_{2,0})\notin [u_{1,-},u_{1,+}]\times [u_{2,-},u_{2,+}],
\end{equation}
where $u_{i,\pm}$ and $\alpha$ stem from Theorem \ref{constantsign}. In particular, \eqref{nontrapped} entails $u_{i,0}\neq 0$, i.e., $(u_{1,0},u_{2,0})$ is nontrivial.
\end{thm}
\begin{proof}
Fix $\hat{R}>0$ so large that assertion (3) of Lemma \ref{degree2} holds and, moreover,
\begin{equation}\label{choice}
[u_{1,-},u_{1,+}]\times [u_{2,-},u_{2,+}]\subseteq B_{\hat{R}}.
\end{equation}
Then choose $\tilde{R}>\hat{R}$ fulfilling assertion (1) in Lemma \ref{degree1}. Due to \eqref{homo1} and \eqref{homo2} one has
\begin{equation}\label{degest3}
\mathcal{H}(1,u_1,u_2) = \mathcal{K}(1,u_1,u_2)\quad\forall\, (u_1,u_2)\in \overline{B}_{\hat{R}}.
\end{equation}
Since both $\mathcal{H}(1,\cdot,\cdot)\lfloor_{\partial B_{\tilde{R}}}$ and $\mathcal{K}(1,\cdot,\cdot)\lfloor_{\partial B_{\hat{R}}}$ do not vanish, \eqref{degest3} ensures that the number $\deg(\mathcal{H}(1,\cdot,\cdot),B_{\tilde{R}}\setminus \overline{B}_{\hat{R}},0)$ is well defined. By excision and domain additivity we thus arrive at
\begin{equation*}
%\label{degest4}
\begin{split}
&\deg(\mathcal{H}(1,\cdot,\cdot),B_{\tilde{R}},0) = \deg(\mathcal{H}(1,\cdot,\cdot),B_{\tilde{R}} \setminus \partial B_{\hat{R}},0) \\
&= \deg(\mathcal{H}(1,\cdot,\cdot),B_{\hat{R}},0) + \deg(\mathcal{H}(1,\cdot,\cdot),B_{\tilde{R}} \setminus \overline{B}_{\hat{R}},0).
\end{split}
\end{equation*}
Bearing in mind \eqref{degest1}, \eqref{degest3}, and \eqref{degest2}, this entails
\begin{equation*}
\begin{split}
\deg(\mathcal{H}(1,\cdot,\cdot),B_{\tilde{R}} \setminus \overline{B}_{\hat{R}},0) &= \deg(\mathcal{H}(1,\cdot,\cdot),B_{\tilde{R}},0) - \deg(\mathcal{H}(1,\cdot,\cdot),B_{\hat{R}},0) \\
&= - \deg(\mathcal{K}(1,\cdot,\cdot),B_{\hat{R}},0) = -1.
\end{split}
\end{equation*}
Therefore, there exists 
\begin{equation}\label{location}
(u_{1,0},u_{2,0})\in B_{\tilde{R}}\setminus\overline{B}_{\hat{R}}    
\end{equation}
such that $\mathcal{H}(1,u_{1,0},u_{2,0})=0$. Clearly, $(u_{1,0},u_{2,0})$ solves \eqref{prob}, while standard results from nonlinear regularity theory (cf. the proof of Theorem \ref{constantsign}) yield $(u_{1,0},u_{2,0})\in C^{1,\alpha}(\overline{\Omega})^2$. Finally, \eqref{choice} and \eqref{location} directly lead to \eqref{nontrapped}.
\end{proof}
\section*{Acknowledgments}
U.Guarnotta and S.A. Marano were supported by the following research projects: 1) PRIN 2017 `Nonlinear Differential Problems via Variational, Topological and Set-valued Methods' (Grant No. 2017AYM8XW) of MIUR; 2) `MO.S.A.I.C.' PRA 2020--2022 `PIACERI' Linea 2 (S.A. Marano) and Linea 3 (U. Guarnotta)  of the University of Catania. U. Guarnotta also acknowledges the support of the GNAMPA-INdAM Project CUP\_E55F22000270001. A. Moussaoui was supported by the Directorate-General of Scientific Research and Technological Development (DGRSDT).

\begin{thebibliography}{99}
%
\bibitem{AH}
W. Allegretto and Y.X. Huang, \textit{A Picone's identity for the $p$-Laplacian and applications,} Nonlinear Anal. \textbf{32} (1998), 819–830.
%
\bibitem{Am}
P. Amster, \textit{Multiple solutions for an elliptic system with indefinite Robin boundary conditions,} Adv. Nonlinear Anal. \textbf{8} (2019), 603--614.
%
\bibitem{CLM}
S. Carl, V. K. Le, and D. Motreanu, \textit{Nonsmooth variational problems and their inequalities. Comparison principles and applications,} Springer Monographs in Mathematics, Springer, New York, 2007.
%
\bibitem{CM}
S. Carl and D. Motreanu, \textit{Extremal solutions for nonvariational quasilinear elliptic systems via expanding trapping regions}, Monatsh. Math. \textbf{182} (2017), 801–821.
%
\bibitem{DM}
H. Didi and A. Moussaoui, \textit{Multiple positive solutions for a class of quasilinear singular elliptic systems,} Rend. Circ. Mat. Palermo (2) \textbf{69} (2020), 977--994.
%
\bibitem{F}
J. Franc${\rm\mathring{u}}$, \textit{Monotone operators. A survey directed to applications to differential equations}, Apl. Mat. \textbf{35} (1990), 257–301.
%
\bibitem{GP}
L. Gasi\'nski and N.S. Papageorgiou, \textit{Exercises in analysis. Part 2. Nonlinear analysis}, Problem Books in Mathematics, Springer, Cham, 2016.
%
\bibitem{GLM}
U. Guarnotta, S.A. Marano, and R. Livrea, \textit{Some recent results on singular $p$-Laplacian systems,} http://arxiv.org/abs/2207.02452.
%
%\bibitem{K}
%S. Kakutani, \textit{A proof of Schauder's theorem}, J. Math. Soc. Japan \textbf{3} (1951), 228–231.
%
\bibitem{Le}
An L\^{e}, \textit{Eigenvalue problems for the $p$-Laplacian,} Nonlinear Anal. \textbf{64} (2006), 1057–1099.
%
\bibitem{L}
G.M. Lieberman, \textit{Boundary regularity for solutions of degenerate elliptic equations}, Nonlinear Anal. \textbf{12} (1988), 1203–1219.
%
\bibitem{Lind}
P. Lindqvist, \textit{Notes on the $p$-Laplace equation}, Report, University of Jyväskylä Department of Mathematics and Statistics \textbf{102}, University of Jyväskylä, Jyväskylä, 2006.
%
\bibitem{Mo}
D. Motreanu, \textit{Nonlinear differential problems with smooth and nonsmooth constraints,} Math. Anal. Appl. Ser., Academic Press, London, 2018.
%
\bibitem{MMP}
D. Motreanu, V.V. Motreanu, and N.S. Papageorgiou, \textit{Topological and variational methods with applications to nonlinear boundary value problems}, Springer, New York, 2014.
%
\bibitem{MMPe}
D. Motreanu, A. Moussaoui, and D.S. Pereira, \textit{Multiple solutions for nonvariational quasilinear elliptic systems}, Mediterr. J. Math. \textbf{15} (2018), Paper no. 88, 14 pp.
%
\bibitem{PW}
N.S. Papageorgiou and P. Winkert, \textit{Nonlinear Robin problems with a reaction of arbitrary growth}, Ann. Mat. Pura Appl. \textbf{195} (2016), 1207–1235.
%
\bibitem{PS}
P. Pucci and J. Serrin, \textit{The maximum principle}, Prog. Nonlinear Differential Equations Appl. {\bf 73}, Birkh\"auser Verlag, Basel, 2007.
%
\bibitem{Zeidler}
E. Zeidler, \textit{Nonlinear functional analysis and its applications. II/B. Nonlinear monotone operators,} Springer-Verlag, New York, 1990.
\end{thebibliography}
\end{document}